\documentclass[11pt,a4paper]{amsart}

\usepackage[T1]{fontenc}
\usepackage{enumitem}
\usepackage{lmodern}
\usepackage{microtype}
\usepackage[textwidth=160mm,textheight=238mm,centering]{geometry}
\usepackage{amssymb}
\usepackage{xcolor}
\usepackage[colorlinks=true,linkcolor=blue!45!black,citecolor=blue!45!black,urlcolor=blue!45!black]{hyperref}

\hypersetup{
  pdftitle={Continuity for random walks in hyperbolic percolation: drift, entropy, and dimension},
  pdfauthor={Kohki Sakamoto}
}

\numberwithin{equation}{section}

\newtheorem{theorem}{Theorem}[section]
\newtheorem{proposition}[theorem]{Proposition}
\newtheorem{lemma}[theorem]{Lemma}
\newtheorem{corollary}[theorem]{Corollary}

\theoremstyle{definition}
\newtheorem{definition}[theorem]{Definition}
\newtheorem{question}[theorem]{Question}

\theoremstyle{remark}
\newtheorem{remark}[theorem]{Remark}

\newcommand{\G}{\Gamma}
\newcommand{\Z}{\mathbb Z}
\newcommand{\Om}{\Omega}
\newcommand{\hbdG}{\partial_h\Gamma}
\newcommand{\Prob}{\operatorname{Prob}}
\newcommand{\1}{\mathbf 1}
\newcommand{\dd}{\mathop{}\!\mathrm d}
\newcommand{\eps}{\varepsilon}
\newcommand{\whmu}{\widehat\mu}
\newcommand{\bnu}{\boldsymbol\nu}
\newcommand{\bnd}{\operatorname{bnd}}

\title[Continuity in hyperbolic percolation]{Continuity for random walks in hyperbolic percolation: drift, entropy, and dimension}
\author{Kohki Sakamoto}
\address{Graduate School of Mathematical Sciences, The University of Tokyo, 3-8-1 Komaba, Meguro, Tokyo 153-8914, Japan}
\email{sakamoto-kohki571@g.ecc.u-tokyo.ac.jp}

\begin{document}

\begin{abstract}
For simple random walk on infinite clusters of Bernoulli bond percolation on Cayley graphs of hyperbolic groups, we prove that the drift and asymptotic entropy depend continuously on the percolation parameter throughout the supercritical phase. Together with the dimension formula, this implies continuity of the Hausdorff dimension of harmonic measure on the Gromov boundary.
\end{abstract}

\maketitle

\section{Introduction}\label{sec:introduction}
In their seminal paper \cite{BLS1999}, Benjamini, Lyons, and Schramm studied simple random walk on infinite clusters of Bernoulli bond percolation on Cayley graphs and asked which properties of the walk on the ambient graph are inherited by its infinite percolation clusters. For a nonamenable Cayley graph, simple random walk has positive drift and positive asymptotic entropy, and hence the graph admits non-constant bounded harmonic functions. They showed that simple random walk on every infinite percolation cluster also has positive drift and positive asymptotic entropy (defined below) \cite[Theorem~4.4, Lemma~4.6]{BLS1999}.

Their result leads to the problem of understanding how the drift and asymptotic entropy vary with the percolation parameter. Indeed, Benjamini, Lyons, and Schramm conjectured that the asymptotic entropy is monotone in $p$ \cite[Conjecture~4.10]{BLS1999} (see \cite{LyonsWhite2023} for related progress). Another natural question concerns the regularity of the drift and asymptotic entropy as functions of the percolation parameter. In this paper, we investigate this question when the underlying graph is Gromov hyperbolic.

To state our result precisely, let $\Gamma$ be a non-elementary hyperbolic group and let $G$ be a Cayley graph of $\Gamma$. Let $p_c$ denote the \emph{critical probability} for the existence of an infinite open cluster. For $p>p_c$, let $\mu_p^\infty$ be the law of Bernoulli bond percolation conditioned on the open cluster $C_p(\omega)$ of the identity being infinite. Given such an environment, let $P_p^\omega$ denote the quenched law of simple random walk on $C_p(\omega)$, started at the identity. Then there exist deterministic constants $\ell(p),h(p)>0$ such that, for $\mu_p^\infty$-almost every $\omega$ and $P_p^\omega$-almost every path $x=(x_n)_{n\ge0}$,
\[
  \frac{|x_n|}{n}\to \ell(p), \quad -\frac1n\log P_p^\omega(X_n=x_n)\to h(p).
\]
Here $|\cdot|$ denotes word length in the ambient Cayley graph. The constants $\ell(p)$ and $h(p)$ are called the \emph{drift} and the \emph{asymptotic entropy}, respectively.
\begin{theorem}\label{thm:main}
Let $\Gamma$ be a non-elementary hyperbolic group. For Bernoulli bond percolation on any Cayley graph of $\Gamma$, the functions
\[
  p\mapsto \ell(p), \quad p\mapsto h(p)
\]
are continuous on $(p_c,1]$.
\end{theorem}

Bernoulli percolation on Gromov hyperbolic graphs has been studied extensively. Important early examples include tessellations of $\mathbb H^2$ and Cayley graphs of cocompact Fuchsian groups, studied in particular by Lalley and by Benjamini--Schramm \cite{Lalley1998,BenjaminiSchramm2001,Lalley2001}. One characteristic feature of these models is the presence of two nondegenerate supercritical phases. Let $p_u$ denote the \emph{uniqueness threshold}, that is, the infimum of the parameters $p$ for which there is almost surely a unique infinite cluster. On $\mathbb Z^d$ one has $p_u=p_c$, while on a regular tree one has $p_u=1$. Thus one of the two possible supercritical phases degenerates in each of these classical examples. By contrast, every one-ended hyperbolic Cayley graph satisfies
\[
  p_c<p_u<1.
\]
Indeed, hyperbolic groups are finitely presented, so the result of Babson and Benjamini \cite{BabsonBenjamini1999} gives $p_u<1$, while Hutchcroft proved $p_c<p_u$ for every nonamenable quasi-transitive Gromov hyperbolic graph \cite{Hutchcroft2019}. Theorem~\ref{thm:main} gives continuity across this uniqueness threshold.

Further, in the hyperbolic setting, the Gromov boundary offers a geometric perspective on this phase transition via the limit sets of clusters. The \emph{limit set} $\Lambda(C)\subset\partial\Gamma$ of an infinite cluster $C$ consists of its accumulation points in the Gromov compactification. For site percolation on certain planar Cayley graphs of cocompact Fuchsian groups, Lalley showed that, in the coexistence phase, the Hausdorff dimension of the limit set is a continuous, strictly increasing function of $p$, tending to $0$ at $p_c$ and to $1$ at $p_u$ \cite{Lalley2001}. By contrast, Hutchcroft and Pan proved that, for percolation on the product of a regular tree and an infinite amenable Cayley graph, the Hausdorff dimension of the limit set in the boundary of the tree factor jumps at $p_u$ \cite{HutchcroftPan2024}. Although the product is not hyperbolic, this result suggests that, beyond the planar case, the dimension of the limit set may exhibit discontinuity. 

The harmonic measure of the cluster is another important object on the boundary. Rather than recording every direction reached by the infinite cluster, it records the typical direction of a random-walk trajectory. In the present setting, the walk converges almost surely to the Gromov boundary $\partial\Gamma$, and we denote the distribution of its limit point by $\nu_\omega^p$. The dimension formula for harmonic measure \cite[Theorem~1.1]{Sakamoto2024} gives, for every visual metric $d_\eps$ with parameter $\eps$,
\begin{equation}\label{eq:dimension-formula-intro}
  \dim_H^{d_\eps}(\nu_\omega^p) =\frac{h(p)}{\eps\ell(p)}
\end{equation}
for $\mu_p^\infty$-almost every $\omega$. We therefore obtain the following consequence.

\begin{corollary}\label{cor:dimension-intro}
For $p>p_c$, let $D_\varepsilon(p)$ denote the almost sure value of the Hausdorff dimension of $\nu_\omega^p$ with respect to a visual metric $d_\varepsilon$. Then the function $p\mapsto D_\varepsilon(p)$ is continuous on $(p_c,1]$.
\end{corollary}

In particular, the Hausdorff dimension of the harmonic measure is continuous even at the uniqueness threshold $p_u$.

For comparison, a similar distinction between a geometric limit set and a limiting boundary measure is already known for branching Brownian motion on $\mathbb H^2$ with branching rate $\beta$. Lalley and Sellke showed that the set of all boundary accumulation points undergoes a phase transition at the recurrence-transience threshold $\beta=1/8$ \cite{LalleySellke1997,Lalley2006}. Here, this threshold plays a role analogous to the uniqueness threshold $p_u$ above. Geldbach instead studied the normalized empirical distribution of the particles and proved that its limiting boundary measure has Hausdorff dimension $\min\{2\beta,1\}$, which is continuous in $\beta$ \cite{Geldbach2025}.

\subsection{Related work}\label{subsect:relatedworks}
A parallel problem concerns perturbations of the increment law while keeping the underlying graph fixed. Erschler raised questions about the continuity of range, entropy, and drift in this setting, and Erschler and Kaimanovich proved continuity of entropy and drift on hyperbolic groups under some moment assumptions \cite{Erschler2011,ErschlerKaimanovich2013}. For measures with fixed finite support, Ledrappier obtained Lipschitz regularity, Mathieu established differentiability formulas for symmetric measures, and Gou\"ezel proved real analyticity of entropy and drift \cite{Ledrappier2013,Mathieu2015,Gouezel2017}. Related continuity and differentiability results for groups with contracting behavior were obtained by Mathieu and Sisto \cite{MathieuSisto2020} and Choi \cite{Choi2026}. Outside hyperbolic-like settings, Silva proved continuity of asymptotic entropy for random walks on a broad class of wreath products \cite{Silva2026}. See also the references therein for further historical background.

Gu and Zhao \cite{GuZhao2025} recently proved that the diffusivity and conductivity of supercritical Bernoulli percolation on $\mathbb Z^d$ are $C^\infty$ functions of $p$. Although their setting and methods are quite different from ours, their result was one of the motivations for the present question and suggests that stronger regularity in the percolation parameter may also hold in the hyperbolic setting.

Another related perspective comes from random Schreier graphs. Bowen \cite{Bowen2014} identified the Furstenberg entropy of stationary spaces arising from invariant random subgroups with the averaged asymptotic entropy of random walks on the associated Schreier graphs, and proved continuity of this quantity for certain families of tree-like Schreier graphs with controlled return-time probabilities.

\subsection{Questions}

\begin{question}\label{q:regularity}
Are $\ell(p)$ and $h(p)$ differentiable, smooth, or analytic in the supercritical phase?
\end{question}

Discontinuity phenomena are known when one varies the driving measure of a random walk (see, for example, \cite{Erschler2011}). This raises the analogous question for the percolation parameter.

\begin{question}
Does there exist a finitely generated group $\Gamma$ and a Cayley graph of $\Gamma$ for which $p\mapsto \ell(p)$ or $p\mapsto h(p)$ is discontinuous at some supercritical parameter?
\end{question}

\subsection{Outline of the proof}

Our approach is based on boundary theory. Boundary methods have played a central role in the parallel problem of how drift and entropy vary with the driving measure, from the continuity results of Erschler and Kaimanovich \cite{ErschlerKaimanovich2013} to the stronger regularity results of Ledrappier \cite{Ledrappier2013} and Gou\"ezel \cite{Gouezel2017}.

For the drift, our argument is inspired by that of Erschler and Kaimanovich \cite{ErschlerKaimanovich2013}. The starting point is a Furstenberg-type formula expressing the drift in terms of the Busemann cocycle and a measure on the horoboundary which is stationary with respect to the walk \cite{Furstenberg1963,Furstenberg1971}. Since the walk depends on the environment in our problem, we introduce the notion of \emph{stationary families}, measurable families of boundary measures satisfying a suitable relation defined in terms of rerooting. We establish the formula for every stationary family on the horoboundary (Theorem~\ref{thm:furstenberg}). For $p_n\to p>p_c$, compactness gives subsequential limits of $p_n$-stationary families, which we show are $p$-stationary. Passing to the limit in the formula then gives continuity of $\ell(p)$ (Theorem~\ref{thm:drift-continuity}).

For the entropy, the argument of Erschler and Kaimanovich requires uniform control of walks conditioned on their limit points. Since such control seems more difficult to obtain in random environments, we instead follow a softer approach recently developed by Silva \cite[Section~8]{Silva2026}, based on the Kaimanovich--Vershik formula and lower semicontinuity of Kullback--Leibler divergence. We establish the formula in our setting (Theorem~\ref{thm:kv}), expressing $h(p)$ as an average of these divergences between harmonic measures on the Gromov boundary and their rerooted versions. Unlike the drift formula, it involves the hitting measures themselves, so subsequential limits must be identified with the harmonic family at the limiting parameter. We do this by proving uniqueness of the stationary family on the Gromov boundary (Theorem~\ref{thm:uniqueness-gromov}).

For random walks on hyperbolic groups with a fixed driving measure, uniqueness follows from contraction of the boundary action and non-atomicity of the stationary measure. In our setting, rerooting changes the boundary measure, so this argument does not apply directly. We use a bilateral walk and apply boundary contraction outside neighborhoods of its backward limit viewed from the current position. Stationarity and non-atomicity of harmonic measure provide the required mass estimates, which, together with a martingale argument, yield uniqueness. Compactness then gives weak continuity of the corresponding joint measures (Corollary~\ref{cor:harmonic-family-continuity}).

Combining this continuity with lower semicontinuity of Kullback--Leibler divergence gives lower semicontinuity of $h(p)$ (Proposition~\ref{prop:entropy-lsc}). Upper semicontinuity follows from continuity and subadditivity of the averaged finite-time entropies (Proposition~\ref{prop:entropy-usc}).

\subsection*{Acknowledgments}
The author would like to thank Tom Hutchcroft and Naotaka Kajino for comments that motivated the question studied in this paper, and J\'er\'emie Brieussel and Ryokichi Tanaka for their helpful comments. This work was supported by JSPS KAKENHI Grant Number 25KJ0862 and by the FoPM WINGS Program at the University of Tokyo.

\newpage
\section{Preliminaries}\label{sec:preliminaries}

\subsection{Hyperbolic groups and boundaries}
Let $S$ be a finite symmetric generating set of a finitely generated group $\G$, and let $d=d_S$ be the word metric. We write $|g|=d(o,g)$ for $g\in\G$, where $o$ is the identity element. We identify the vertex set of the Cayley graph $G=(V,E)$ with $\G$. The Gromov product is defined by
\[
  (y|z)_x=\frac12\bigl(d(x,y)+d(x,z)-d(y,z)\bigr)
\]
for $x,y,z\in\G$. Throughout the paper, $\G$ is assumed to be non-elementary and \emph{word-hyperbolic}, i.e., it is not virtually cyclic and there exists $\delta\ge0$ such that
\[
  (x|z)_w\ge\min\{(x|y)_w,(y|z)_w\}-\delta
\]
for all $x,y,z,w\in\G$.

\begin{definition}[Gromov boundary]\label{def:gromov-boundary}
A geodesic ray based at $o$ is an isometric embedding $\gamma:\Z_{\ge0}\to G$ with $\gamma(0)=o$. Two such rays are called equivalent if their images have finite Hausdorff distance. The \emph{Gromov boundary} $\partial\Gamma$ is the set of equivalence classes of geodesic rays based at $o$.
\end{definition}

We write $\overline\G=\G\cup\partial\Gamma$ for the Gromov compactification. It is compact and metrizable, and the left action of $\G$ extends continuously to $\overline\G$. For $u,v\in\overline\G$, we extend the Gromov product by
\[
  (u|v)_o=\sup\bigl\{\liminf_{n\to\infty}(x_n|y_n)_o:x_n,y_n\in\G,\ x_n\to u,\ y_n\to v\bigr\},
\]
where the corresponding sequence is taken to be constant when $u$ or $v$ lies in $\G$. The hyperbolicity inequality at basepoint $o$ extends to $\overline\G$, with an additive error depending only on $\delta$. For $z_n\in\overline\G$ and $\zeta\in\partial\Gamma$, we have $z_n\to\zeta$ if and only if $(z_n|\zeta)_o\to\infty$. We refer to \cite{Gromov1987} for these facts.

For every sufficiently small $\eps>0$, there exists a \emph{visual metric} $d_\eps$ on $\partial\Gamma$ such that
\[
  d_\eps(\zeta,\zeta') \asymp e^{-\eps(\zeta|\zeta')_o}.
\]
See \cite[Section~1]{Coornaert1993} for details.

We use the following standard contraction property of the boundary action in the proof of uniqueness (see also \cite[Section~7]{Bowditch1998}).

\begin{proposition}\label{prop:boundary-contraction}
Let $(g_n)_{n\ge1}$ be a sequence in $\G$ such that
\[
  g_no\to a\in\partial\Gamma, \,\, g_n^{-1}o\to b\in\partial\Gamma.
\]
Then, for every neighborhood $U$ of $a$ and every neighborhood $V$ of $b$,
\[
  g_n(\partial\Gamma\setminus V)\subset U
\]
for all sufficiently large $n$.
\end{proposition}

\begin{proof}
Let $K$ be a compact subset of $\partial\Gamma\setminus\{b\}$. Since $g_n^{-1}o\to b$, there is $C<\infty$ such that $(\zeta|g_n^{-1}o)_o\le C$ for every $\zeta\in K$ and all sufficiently large $n$. Otherwise, the hyperbolicity inequality would force a sequence of points in $K$ to converge to $b$. For $x\in\G$, we have
\[
  (g_nx|g_no)_o=|g_n|-(x|g_n^{-1}o)_o.
\]
Passing to the boundary gives this identity up to an additive error depending only on $\delta$. The hyperbolicity inequality therefore yields
\[
  (g_n\zeta|a)_o\ge\min\{|g_n|-C,(g_no|a)_o\}-C_\delta\to\infty
\]
uniformly for $\zeta\in K$, where $C_\delta$ depends only on $\delta$. Thus $g_n\zeta\to a$ uniformly on $K$, proving the claim.
\end{proof}

For the drift formula, we use the horoboundary and the Busemann cocycle on it.

\begin{definition}[Horoboundary]\label{def:horoboundary}
For $y\in\G$, we define the normalized distance function
\[
  h_y(x)=d(x,y)-d(o,y),\quad x\in\G.
\]
The closure of $\{h_y:y\in\G\}$ in $\mathbb R^\G$, with the topology of pointwise convergence, is called the \emph{horofunction compactification}. The \emph{horoboundary} $\hbdG$ is the complement of $\{h_y:y\in\G\}$ in this compactification.
\end{definition}

The horoboundary is compact and metrizable. Every $\xi\in\hbdG$ is a $1$-Lipschitz function with $\xi(o)=0$. The action of $\G$ on $\hbdG$ is given by
\[
  (g\xi)(x)=\xi(g^{-1}x)-\xi(g^{-1}).
\]
If $h_{y_n}\to\xi\in\hbdG$, then $y_n$ converges in the Gromov compactification to a point depending only on $\xi$. This defines a continuous equivariant surjection $\pi_h:\hbdG\to\partial\Gamma$ \cite[Section~4]{WebsterWinchester2005}.

\begin{definition}[Busemann cocycle]\label{def:busemann-cocycle}
The \emph{Busemann cocycle} $\beta\colon \Gamma\times\hbdG\to\mathbb R$ is defined by $\beta(g,\xi)=\xi(g)$. It satisfies the cocycle identity
\[
  \beta(gh,\xi) = \beta(g,\xi)+\beta(h,g^{-1}\xi).
\]
\end{definition}

\begin{lemma}\label{lem:horofunction-asymptotics}
Let $x_n\to\zeta\in\partial\Gamma$, and let $\xi\in\hbdG$ satisfy $\pi_h(\xi)\ne\zeta$. Then
\[
  \beta(x_n,\xi)=|x_n|+O(1),
\]
where the bounded error may depend on the sequence $(x_n)_{n\ge1}$ and on $\xi$.
\end{lemma}

\begin{proof}
We choose $y_k\in\G$ such that $h_{y_k}\to\xi$ pointwise and $y_k\to\pi_h(\xi)$ in the Gromov compactification. Since the two boundary limits are distinct, there is $C<\infty$ such that
\[
  (x_n|y_k)_o\le C
\]
for all sufficiently large $n$ and $k$. Hence
\[
  h_{y_k}(x_n)=d(x_n,y_k)-|y_k|=|x_n|-2(x_n|y_k)_o=|x_n|+O(1)
\]
for all sufficiently large $n$ and $k$. Letting $k\to\infty$ proves the claim.
\end{proof}

\subsection{Bernoulli percolation}
We use the \emph{standard coupling} of Bernoulli percolation to compare different parameters. Let $\Om=[0,1]^E$ be equipped with the product topology and product Lebesgue measure $\mu$. The space $\Om$ is compact and metrizable. For $e\in E$, we write $U_e(\omega)=\omega(e)$. The group acts by left translation, i.e.,
\[
  (g\omega)(e)=\omega(g^{-1}e).
\]
At parameter $p$, the edge $e$ is open when $U_e\le p$. We denote by $C_p(\omega)$ the open cluster of $o$ and set
\[
  \Om_p^\infty=\{\omega:C_p(\omega)\text{ is infinite}\}, \,\, \theta(p)=\mu(\Om_p^\infty).
\]
The critical parameter $p_c$ is defined by $p_c=\inf\{p\in[0,1]:\theta(p)>0\}$. Throughout this paper, we only consider $p>p_c$ and write $\mu_p^\infty=\mu(\,\cdot\mid\Om_p^\infty)$.

For $s\in S$, we set $e_s=\{o,s\}$ and define
\[
  \deg_p(\omega)=\sum_{s\in S}\1_{\{U_{e_s}\le p\}}(\omega), \,\, a_{p,s}(\omega)=\1_{\Om_p^\infty}(\omega)\1_{\{U_{e_s}\le p\}}(\omega).
\]
We define the normalizing constant $Z_p$ by
\[
  Z_p=\int_\Om \deg_p(\omega)\1_{\Om_p^\infty}(\omega)\dd\mu(\omega),
\]
so that $0<\theta(p)\le Z_p\le |S|\theta(p)$. We define the degree-biased environment law by
\begin{equation}\label{eq:muhat}
  \dd\whmu_p(\omega) =\frac{\deg_p(\omega)\1_{\Om_p^\infty}(\omega)}{Z_p}\dd\mu(\omega).
\end{equation}
Note that the measures $\whmu_p$ and $\mu_p^\infty$ are equivalent.

\begin{lemma}\label{lem:environment-continuity}
Let $p_n\to p>p_c$. Then
\[
  \1_{\Om_{p_n}^\infty}\to\1_{\Om_p^\infty}, \quad \deg_{p_n}\1_{\Om_{p_n}^\infty} \to \deg_p\1_{\Om_p^\infty},
\]
and, for every $s\in S$,
\[
  a_{p_n,s}\to a_{p,s}
\]
$\mu$-almost surely and in $L^1(\mu)$. Consequently, $Z_{p_n}\to Z_p$ and $\whmu_{p_n}\to\whmu_p$ in total variation.
\end{lemma}

\begin{proof}
Under the standard coupling, the map $q\mapsto\1_{\Om_q^\infty}(\omega)$ is non-decreasing. Since $\theta$ is continuous at $p$ \cite[Corollary~1.3]{HaggstromPeres1999}, this map is continuous at $p$ for $\mu$-almost every $\omega$. Since $\mu(U_{e_s}=p)=0$ for every $s\in S$, the status of every edge incident to $o$ is eventually the same at $p_n$ and $p$, for $\mu$-almost every $\omega$. Taking products and summing over $S$ gives the asserted convergences almost surely. Uniform boundedness gives convergence in $L^1(\mu)$, and hence $Z_{p_n}\to Z_p$. Since $Z_p>0$, \eqref{eq:muhat} implies convergence of the densities in $L^1(\mu)$, proving convergence in total variation.
\end{proof}

\subsection{Random walks and harmonic measures}

For $\omega\in\Om_p^\infty$ and $s\in S$, we define
\begin{equation}\label{eq:qps}
  q_p(\omega,s) = \frac{\1_{\{U_{e_s}\le p\}}(\omega)}{\deg_p(\omega)}.
\end{equation}
For $g\in C_p(\omega)$, let $P_{p,g}^\omega\in\Prob(\G^{\Z_{\ge0}})$ be the quenched law of simple random walk on $C_p(\omega)$, started at $g$. Its transition probabilities are given by
\[
  P_{p,g}^\omega(X_{n+1}=xs\mid X_n=x) = q_p(x^{-1}\omega,s), \quad x\in C_p(\omega),\ s\in S.
\]
We write $P_p^\omega=P_{p,o}^\omega$.

We use a bilateral version of this walk. We first sample $\omega$ according to $\whmu_p$ and then, conditionally on $\omega$, sample $(X_n)_{n\ge0}$ and $(X_{-n})_{n\ge0}$ independently, both with law $P_p^\omega$. We denote the resulting probability measure on $\Om\times\G^\Z$ by $\Lambda_p$ and define the shift by
\[
  T\bigl(\omega,(X_n)_{n\in\Z}\bigr) = \bigl(X_1^{-1}\omega,(X_1^{-1}X_{n+1})_{n\in\Z}\bigr).
\]

Under the degree-biased law, the environment chain is reversible. Indeed, the joint law of the environment and the first step satisfies
\[
  \whmu_p(\dd\omega)q_p(\omega,s) = \frac1{Z_p}a_{p,s}(\omega)\mu(\dd\omega), \quad s\in S.
\]
Since
\[
  a_{p,s}(\omega) = a_{p,s^{-1}}(s^{-1}\omega)
\]
and $\mu$ is invariant under the action of $\G$, this measure is invariant under the involution
\[
  (\omega,s)\mapsto(s^{-1}\omega,s^{-1}).
\]

\begin{proposition}\label{prop:bilateral}
For every $p>p_c$, the system $(\Om\times\G^\Z,\Lambda_p,T)$ is invertible, measure-preserving, and ergodic. Moreover, there exist deterministic constants $\ell(p),h(p)>0$ such that, for $\mu_p^\infty$-almost every $\omega$ and $P_p^\omega$-almost every path $x=(x_n)_{n\ge0}$,
\begin{equation}\label{eq:regularity}
  \frac{|x_n|}{n}\to\ell(p), \quad -\frac1n\log P_p^\omega(X_n=x_n)\to h(p).
\end{equation}
\end{proposition}

\begin{proof}
Invertibility is immediate from the bilateral construction, while measure preservation follows from reversibility. The proof of \cite[Theorem~5.5]{Sakamoto2024} applies equally to the standard coupling used here and gives ergodicity.

The limits in \eqref{eq:regularity} follow from Kingman's subadditive ergodic theorem and are deterministic by ergodicity. As mentioned in the introduction, positivity of the drift and the asymptotic entropy is due to Benjamini, Lyons, and Schramm \cite[Theorem~4.4, Lemma~4.6]{BLS1999}.
\end{proof}

We next define the harmonic measure and introduce the sublinear tracking property used below. By \eqref{eq:regularity},
\[
  \frac{|X_n|}{n}\to\ell(p)>0,
\]
and since we consider nearest-neighbor walks,
\[
  \frac{d(X_n,X_{n+1})}{n}\to0.
\]
Hence \cite[Theorem~7.2]{Kaimanovich2000} yields convergence to the Gromov boundary and sublinear tracking along a geodesic ray. Conditional on $\omega$, both halves of a $\Lambda_p$-distributed path have law $P_p^\omega$, so the limits
\[
  \bnd^+(X)=\lim_{n\to\infty}X_n, \quad \bnd^-(X)=\lim_{n\to\infty}X_{-n}
\]
exist for $\Lambda_p$-almost every bilateral path.

By measurable selection, we can fix a Borel choice $\zeta\mapsto\gamma_\zeta$ of a geodesic ray from $o$ to $\zeta$. For $t\ge0$, we write $\gamma_\zeta(t)$ for $\gamma_\zeta(\lfloor t\rfloor)$. Then the tracking property can be written as
\begin{equation}\label{eq:tracking}
  \frac1n d\bigl(X_n,\gamma_{\bnd^+(X)}(\ell(p)n)\bigr) \to0
\end{equation}
for $\mu_p^\infty$-almost every $\omega$ and $P_p^\omega$-almost every path $X=(X_n)_{n\ge0}$.

\begin{definition}[Harmonic measure]\label{def:harmonic-measure}
For $\omega\in\Om_p^\infty$, the \emph{harmonic measure} is the hitting measure
\[
  \nu_\omega^p = (\bnd^+)_*P_p^\omega \in\Prob(\partial\Gamma).
\]
\end{definition}

Since the transition probabilities depend measurably on the environment, so do the path laws $P_p^\omega$. The set of paths converging to the Gromov boundary is Borel, and the limit map, extended by a fixed boundary point outside this set, is Borel. Hence $\omega\mapsto\nu_\omega^p$ is measurable. For $s\in S$ with $q_p(\omega,s)>0$, let $\nu_{\omega,s}^p$ denote the conditional law of the boundary point given $X_1=s$. The Markov property and equivariance give
\[
  \nu_{\omega,s}^p=(\bnd^+)_*P_{p,s}^\omega=s_{*}(\bnd^+)_*P_p^{s^{-1}\omega}=s\nu_{s^{-1}\omega}^p,
\]
and hence averaging over the first step gives
\begin{equation}\label{eq:harmonic-stationarity}
  \nu_\omega^p=\sum_{s:\,q_p(\omega,s)>0}q_p(\omega,s)s\nu_{s^{-1}\omega}^p
\end{equation}
for $\mu_p^\infty$-almost every $\omega$.

For the later arguments, we need non-atomicity of harmonic measure. By \cite[Corollary~5.10]{Sakamoto2024}, $\nu_\omega^p$ is exact dimensional with positive dimension $h(p)/(\eps\ell(p))$, and is therefore non-atomic. For completeness, we give a direct proof.

\begin{lemma}\label{lem:harmonic-nonatomic}
For every $p>p_c$, the measure $\nu_\omega^p$ is non-atomic for $\mu_p^\infty$-almost every $\omega$.
\end{lemma}

\begin{proof}
Let
\[
  A=\{(\omega,X)\in\Om\times\G^\Z:\bnd^+(X)=\bnd^-(X)\}.
\]
On $A$, we write $\xi=\bnd^+(X)=\bnd^-(X)$. Applying \eqref{eq:tracking} to both half-paths with the same ray $\gamma_\xi$ gives
\begin{equation}\label{eq:same-boundary-sublinear}
  d(X_{-n},X_n) \le d\bigl(X_{-n},\gamma_\xi(\ell(p)n)\bigr) +d\bigl(X_n,\gamma_\xi(\ell(p)n)\bigr) =o(n).
\end{equation}

By $T$-invariance of $\Lambda_p$,
\[
  d(X_{-n},X_n)=|X_{-n}^{-1}X_n| \ \stackrel{\mathrm{law}}{=}\ |X_{2n}|.
\]
Indeed, after shifting the bilateral path by $n$, the relative position at time $2n$ is $X_{-n}^{-1}X_n$. Since $\ell(p)>0$, \eqref{eq:regularity} gives
\[
  \Lambda_p\left(\frac{d(X_{-n},X_n)}{2n}<\frac{\ell(p)}2\right) =\Lambda_p\left(\frac{|X_{2n}|}{2n}<\frac{\ell(p)}2\right) \to0.
\]
By \eqref{eq:same-boundary-sublinear},
\[
  \1_A\le\liminf_{n\to\infty} \1_{\{d(X_{-n},X_n)/(2n)<\ell(p)/2\}}
\]
$\Lambda_p$-almost surely. Fatou's lemma therefore yields $\Lambda_p(A)=0$.

Conditional on $\omega$, the two boundary limits are independent with law $\nu_\omega^p$. Since they are almost surely distinct, $\nu_\omega^p$ is non-atomic for $\whmu_p$-almost every $\omega$, and hence for $\mu_p^\infty$-almost every $\omega$.
\end{proof}

\section{Stationary boundary families}\label{sec:stationary}

Throughout this section, the symbol $B$ denotes either $\partial\Gamma$ or $\hbdG$. For a probability measure $\lambda$ on $B$, we write $g\lambda=g_*\lambda$. We regard boundary families as probability measures on $\Om\times B$ with first marginal $\mu$, so that families at different parameters belong to the same compact space. To extend them from $\Om_p^\infty$ to all of $\Om$, we fix $\xi_0\in\hbdG$, write $\zeta_0=\pi_h(\xi_0)$, and set
\[
  \delta_B=\begin{cases}
    \delta_{\xi_0},&B=\hbdG,\\
    \delta_{\zeta_0},&B=\partial\Gamma.
  \end{cases}
\]

\subsection{Stationary families}

Let
\[
  \mathcal P_\mu(B)=\{\eta\in\Prob(\Om\times B):(\pi_\Om)_*\eta=\mu\}.
\]
This is convex and compact in the weak topology. For $\eta \in \mathcal P_\mu(B)$, we write its disintegration as
\[
  \eta(\dd\omega,\dd\xi)=\mu(\dd\omega)\eta_\omega(\dd\xi).
\]

\begin{definition}[Stationary family]\label{def:stationary-family}
A measure $\eta\in\mathcal P_\mu(B)$ is called \emph{$p$-stationary} if
\begin{enumerate}[label=(\roman*)]
\item $\eta_\omega=\delta_B$ for $\mu$-almost every $\omega\notin\Om_p^\infty$;
\item for $\mu$-almost every $\omega\in\Om_p^\infty$,
\begin{equation}\label{eq:stationarity-pointwise}
  \eta_\omega=\sum_{s\in S}q_p(\omega,s)s\eta_{s^{-1}\omega}.
\end{equation}
\end{enumerate}
The set of such measures is denoted by $\mathcal S_p(B)$.
\end{definition}

It is straightforward to see that condition~(ii) is equivalent to
\begin{align}
  \int_{\Om_p^\infty}\deg_p(\omega) \int_B F(\omega,\xi)\dd\eta_\omega(\xi)\dd\mu(\omega)= \sum_{s\in S}\int_\Om a_{p,s}(\omega) \int_B F(\omega,s\xi)\dd\eta_{s^{-1}\omega}(\xi)\dd\mu(\omega) \label{eq:integrated-stationarity}
\end{align}
for every bounded Borel function $F$ on $\Om\times B$.

\begin{lemma}\label{lem:weighted-weak-convergence}
Let $B\in\{\partial\Gamma,\hbdG\}$. Suppose $\eta_n\to\eta$ in $\mathcal P_\mu(B)$ and $b_n,b\in L^\infty(\Om,\mu)$ satisfy $b_n\to b$ in $L^1(\mu)$. Then the finite signed measures
\[
  b_n(\omega)\eta_n(\dd\omega,\dd\xi)
\]
converge weakly to $b(\omega)\eta(\dd\omega,\dd\xi)$.
\end{lemma}

\begin{proof}
Let $F\in C(\Om\times B)$. Since all measures have first marginal $\mu$, replacing $b_n$ by $b$ changes the integral by at most $\|F\|_\infty\|b_n-b\|_{L^1(\mu)}$. The claim follows by approximating $b$ in $L^1(\mu)$ by continuous functions and using weak convergence.
\end{proof}

\begin{proposition}\label{prop:stationary-existence}
For every $p>p_c$ and every $B\in\{\partial\Gamma,\hbdG\}$, the set $\mathcal S_p(B)$ is nonempty.
\end{proposition}

\begin{proof}
We first construct an invariant measure with marginal $\whmu_p$. Let
\[
  \mathcal P_{\whmu_p}(B) = \{\lambda\in\Prob(\Om\times B):(\pi_\Om)_*\lambda=\whmu_p\}.
\]
Since $\Om\times B$ is compact and metrizable, this is a nonempty compact convex subset of $\Prob(\Om\times B)$. We define an affine operator $K_p$ by
\[
  \int F\dd(K_p\lambda) = \sum_{s\in S}\int q_p(\omega,s) F(s^{-1}\omega,s^{-1}\xi)\dd\lambda(\omega,\xi)
\]
for bounded Borel functions $F$. By reversibility, $K_p\lambda$ has first marginal $\whmu_p$, so $K_p$ maps $\mathcal P_{\whmu_p}(B)$ into itself.

To prove weak continuity of $K_p$, we extend $q_p(\cdot,s)$ to all of $\Om$ by
\[
  \widetilde q_{p,s}(\omega) = \frac{\1_{\{U_{e_s}\le p\}}(\omega)} {\max\{\deg_p(\omega),1\}}.
\]
Each $\widetilde q_{p,s}$ is bounded, agrees with $q_p(\cdot,s)$ for $\whmu_p$-almost every $\omega$, and is continuous outside the $\whmu_p$-null set $\bigcup_{t\in S}\{\omega:U_{e_t}(\omega)=p\}$. Since all measures in $\mathcal P_{\whmu_p}(B)$ have first marginal $\whmu_p$, the Portmanteau theorem implies that $K_p$ is weakly continuous. Then Schauder's fixed-point theorem gives a fixed point $\lambda\in\mathcal P_{\whmu_p}(B)$. We write
\[
  \lambda(\dd\omega,\dd\xi) = \whmu_p(\dd\omega)\lambda_\omega(\dd\xi)
\]
for its disintegration.
For every bounded Borel function $f$ on $\Om$ and every $\phi\in C(B)$, the fixed-point identity and reversibility give
\[
  \int_\Om f(\omega) \int_B\phi(\xi)\dd\lambda_\omega(\xi) \dd\whmu_p(\omega) = \sum_{s\in S}\int_\Om f(\omega)q_p(\omega,s) \int_B\phi(s\xi)\dd\lambda_{s^{-1}\omega}(\xi) \dd\whmu_p(\omega).
\]
Taking $\phi$ in a countable dense subset of $C(B)$ yields
\[
  \lambda_\omega = \sum_{s\in S}q_p(\omega,s)s\lambda_{s^{-1}\omega}
\]
for $\whmu_p$-almost every $\omega$. Since $\whmu_p$ and $\mu$ are equivalent on $\Om_p^\infty$, we obtain an element of $\mathcal S_p(B)$ by extending the family by $\delta_B$ outside $\Om_p^\infty$.
\end{proof}

Next, we show that stationary families are closed under weak limits when the percolation parameter varies.

\begin{lemma}\label{lem:stationary-stability}
Suppose $p_{n} \to p>p_{c}$, $\eta_n\in\mathcal S_{p_n}(B)$, and $\eta_n\to\eta$ weakly in $\mathcal P_\mu(B)$. Then $\eta\in\mathcal S_p(B)$.
\end{lemma}

\begin{proof}
By Lemmas~\ref{lem:environment-continuity} and~\ref{lem:weighted-weak-convergence}, passing to the limit in
\[
  \1_{\Om\setminus\Om_{p_n}^\infty}\eta_n = \1_{\Om\setminus\Om_{p_n}^\infty}(\mu\otimes\delta_B)
\]
gives $\eta_\omega=\delta_B$ for $\mu$-almost every $\omega\notin\Om_p^\infty$.

We pass to the limit in \eqref{eq:integrated-stationarity} for $F\in C(\Om\times B)$. The coefficients $\deg_{p_n}\1_{\Om_{p_n}^\infty}$ and $a_{p_n,s}$ converge in $L^1(\mu)$ by Lemma~\ref{lem:environment-continuity}. On the right-hand side, the change of variables $\omega=s\omega'$ turns the integrand into
\[
  a_{p_n,s}(s\omega')F(s\omega',s\xi),
\]
so Lemma~\ref{lem:weighted-weak-convergence} applies. The same lemma applies to the left-hand side. Hence \eqref{eq:integrated-stationarity} holds at parameter $p$, which gives $\eta\in\mathcal S_p(B)$.
\end{proof}

\subsection{Uniqueness on the Gromov boundary}

We extend the harmonic family to all of $\Om$ by setting
\[
  \nu_\omega^p=\delta_{\zeta_0} \quad \text{for}\,\,\omega\notin\Om_p^\infty,
\]
and define
\[
  \bnu_p(\dd\omega,\dd\xi) = \mu(\dd\omega)\nu_\omega^p(\dd\xi).
\]
By \eqref{eq:harmonic-stationarity}, we have $\bnu_p\in\mathcal S_p(\partial\Gamma)$. In this subsection, we will show that this is the unique element of $\mathcal S_p(\partial\Gamma)$.

We first prove a deterministic contraction lemma. We fix a compatible metric $d_{\partial\Gamma}$ on $\partial\Gamma$ and a countable dense subset $\mathcal D\subset C(\partial\Gamma)$.

\begin{lemma}\label{lem:moving-repeller}
Let $(g_n)_{n\ge1}$ be a sequence in $\G$, let $\alpha,\gamma\in\partial\Gamma$ be distinct, and let $b_n\in\partial\Gamma$ for $n\ge1$. Assume that $g_no\to\alpha$ and that $g_nb_n=\gamma$ for every $n\ge1$. Then for every neighborhood $U$ of $\alpha$ and every $r>0$,
\[
  g_n\bigl(\partial\Gamma\setminus B_{d_{\partial\Gamma}}(b_n,r)\bigr)\subset U
\]
for all sufficiently large $n$.
\end{lemma}

\begin{proof}
Suppose that the conclusion fails. After passing to a subsequence, we can choose $z_n\in\partial\Gamma$ such that
\[
  d_{\partial\Gamma}(z_n,b_n)\ge r, \,\, g_nz_n\notin U.
\]
By compactness, we may pass to a further subsequence such that
\[
  z_n\to z, \,\, b_n\to c, \,\, g_n^{-1}o\to c_0\in\partial\Gamma.
\]
Since $g_no\to\alpha\in\partial\Gamma$, we have $|g_n|\to\infty$, so the last limit lies in the boundary. By Proposition~\ref{prop:boundary-contraction}, the maps $g_n$ converge locally uniformly to the constant map $\alpha$ on $\partial\Gamma\setminus\{c_0\}$. If $c\ne c_0$, this convergence and $b_n\to c$ give $g_nb_n\to\alpha$, contradicting $g_nb_n=\gamma\ne\alpha$. Thus $c=c_0$.

Since $d_{\partial\Gamma}(z_n,b_n)\ge r$, we have $d_{\partial\Gamma}(z,c_0)\ge r$, so $z\ne c_0$. The same convergence gives $g_nz_n\to\alpha$, contradicting $g_nz_n\notin U$.
\end{proof}

For the bilateral path $X=(X_n)_{n\in\Z}$, let
\[
  \xi^+=\bnd^+(X), \quad \xi^-=\bnd^-(X).
\]
By construction, conditionally on $\omega$, the limit points $\xi^+$ and $\xi^-$ are independent and both have law $\nu_\omega^p$. Thus Lemma~\ref{lem:harmonic-nonatomic} gives $\xi^+\ne\xi^-$ almost surely.

For $n\in\Z$, we set
\[
  \omega_n=X_n^{-1}\omega, \,\, \rho_n=X_n^{-1}\xi^-.
\]
Thus $\rho_n$ is the backward limit viewed from $X_n$, and $X_n\rho_n=\xi^-$. By $T$-invariance, for every $n\in\Z$ the joint law of $(\omega_n,\rho_n)$ is
\begin{equation}\label{eq:repeller-law}
  \whmu_p(\dd\omega)\nu_\omega^p(\dd\zeta).
\end{equation}

To apply Lemma~\ref{lem:moving-repeller} along the walk, we control the mass that $\eta_{\omega_n}$ assigns near $\rho_n$ for an arbitrary stationary family $\eta$. By \eqref{eq:repeller-law}, the expected mass of $B_{d_{\partial\Gamma}}(\rho_n,r)$ is independent of $n$. Non-atomicity of harmonic measure implies that this expectation tends to zero as $r\downarrow0$.

\begin{lemma}\label{lem:small-mass-repeller}
Let $\eta\in\mathcal S_p(\partial\Gamma)$ and set
\[
  \alpha_\eta(r) =\int_{\Om_p^\infty}\int_{\partial\Gamma} \eta_\omega\bigl(B_{d_{\partial\Gamma}}(\zeta,r)\bigr) \dd\nu_\omega^p(\zeta)\dd\whmu_p(\omega).
\]
Then $\alpha_\eta(r)\to0$ as $r\downarrow0$, and
\begin{equation}\label{eq:small-mass-uniform}
  \mathbb E_{\Lambda_p}\!\left[\eta_{\omega_n}\bigl(B_{d_{\partial\Gamma}}(\rho_n,r)\bigr)\right]=\alpha_\eta(r)
\end{equation}
for every $n\in\Z$ and $r>0$.
\end{lemma}

\begin{proof}
For each $\omega$, the set of atoms of $\eta_\omega$ is countable, and Lemma~\ref{lem:harmonic-nonatomic} implies
\[
  \eta_\omega(\{\zeta\})=0
\]
for $\whmu_p(\dd\omega)\nu_\omega^p(\dd\zeta)$-almost every pair $(\omega,\zeta)$. Since
\[
  \eta_\omega\bigl(B_{d_{\partial\Gamma}}(\zeta,r)\bigr) \downarrow\eta_\omega(\{\zeta\}),
\]
dominated convergence gives $\alpha_\eta(r)\to0$. The identity \eqref{eq:small-mass-uniform} follows from \eqref{eq:repeller-law}.
\end{proof}

\begin{lemma}\label{lem:stationary-boundary-contraction}
Let $\eta\in\mathcal S_p(\partial\Gamma)$. Then, for $\mu_p^\infty$-almost every $\omega$ and $P_p^\omega$-almost every path $(X_n)_{n\ge0}$,
\[
X_n\eta_{X_n^{-1}\omega}\to\delta_{\bnd^+(X)}
\]
weakly on $\partial\Gamma$.
\end{lemma}

\begin{proof}
Let $f\in\mathcal D$. We prove that
\begin{equation}\label{eq:martingale-target}
  \int_{\partial\Gamma}f(X_n\zeta)\dd\eta_{X_n^{-1}\omega}(\zeta) \to f(\xi^+)
\end{equation}
in $L^1(\Lambda_p)$ and almost surely. Since $\omega_n$ has law $\whmu_p$ for every $n\ge0$, the stationarity equation \eqref{eq:stationarity-pointwise} holds almost surely at every $\omega_n$.

We set
\[
  M_n(f)=\int_{\partial\Gamma}f(X_n\zeta)\dd\eta_{\omega_n}(\zeta).
\]
The process $(M_n(f))_{n\ge0}$ is a bounded martingale with respect to the forward filtration $\mathcal F_n^+=\sigma(\omega,X_0,\ldots,X_n)$. Indeed, conditional on $\mathcal F_n^+$, the next position is $X_ns$ with probability $q_p(\omega_n,s)$, and therefore
\begin{align*}
  \mathbb E_{\Lambda_p}[M_{n+1}(f)\mid\mathcal F_n^+] &=\sum_{s\in S}q_p(\omega_n,s) \int f(X_ns\zeta)\dd\eta_{s^{-1}\omega_n}(\zeta)\\
  &=\int f(X_n\zeta)\dd\eta_{\omega_n}(\zeta) =M_n(f).
\end{align*}
Hence $M_n(f)$ converges almost surely to a bounded random variable $M_\infty(f)$.

We now show that $M_\infty(f)=f(\xi^+)$ almost surely. We set $L=\|f\|_\infty$ and fix $\varepsilon>0$. For $r>0$, we define
\[
  E_n(r)=\left\{ \sup_{d_{\partial\Gamma}(z,\rho_n)\ge r} |f(X_nz)-f(\xi^+)|\ge\varepsilon \right\}.
\]
Then Lemma~\ref{lem:moving-repeller} gives $\1_{E_n(r)}\to0$ almost surely, and hence $\Lambda_p(E_n(r))\to0$. On $E_n(r)^c$,
\[
  |M_n(f)-f(\xi^+)| \le\varepsilon+2L\eta_{\omega_n}\bigl(B_{d_{\partial\Gamma}}(\rho_n,r)\bigr).
\]
On $E_n(r)$, the difference is at most $2L$. Therefore, taking expectations and using Lemma~\ref{lem:small-mass-repeller} gives
\[
  \mathbb E_{\Lambda_p}|M_n(f)-f(\xi^+)| \le \varepsilon+2L\alpha_\eta(r)+2L\Lambda_p(E_n(r)).
\]
Letting $n\to\infty$, then $r\downarrow0$ and $\varepsilon\downarrow0$, proves that $M_n(f)\to f(\xi^+)$ in $L^1(\Lambda_p)$. Since $M_n(f)$ also converges almost surely to $M_\infty(f)$, the two limits agree almost surely. This proves \eqref{eq:martingale-target}.

Since $\mathcal D$ is countable and dense, we have
\[
  X_n\eta_{\omega_n}\to\delta_{\xi^+}
\]
almost surely under $\Lambda_p$. Disintegrating the forward marginal $\whmu_p(\dd\omega)P_p^\omega(\dd X)$ and using $\whmu_p\sim\mu_p^\infty$ proves the claim.
\end{proof}
The preceding contraction lemma yields uniqueness, as in the classical case of random walks on hyperbolic groups driven by a fixed measure.

\begin{theorem}\label{thm:uniqueness-gromov}
For every $p>p_c$, the harmonic family $\bnu_p$ is the unique element of $\mathcal S_p(\partial\Gamma)$.
\end{theorem}

\begin{proof}
Let $\eta\in\mathcal S_p(\partial\Gamma)$. For $f\in\mathcal D$, the martingale from the proof of Lemma~\ref{lem:stationary-boundary-contraction} satisfies
\[
  M_0(f)=\mathbb E_{\Lambda_p}[M_n(f)\mid\omega].
\]
Since $M_n(f)\to f(\xi^+)$ in $L^1(\Lambda_p)$,
\[
  \int_{\partial\Gamma}f\dd\eta_\omega =\mathbb E_{\Lambda_p}[f(\xi^+)\mid\omega] =\int_{\partial\Gamma}f\dd\nu_\omega^p
\]
for $\whmu_p$-almost every $\omega$. Since $\mathcal D$ is countable and dense, $\eta_\omega=\nu_\omega^p$ for $\mu$-almost every $\omega\in\Om_p^\infty$. Outside $\Om_p^\infty$, both families are $\delta_{\zeta_0}$ by definition.
\end{proof}

\begin{corollary}\label{cor:harmonic-family-continuity}
If $p_n\to p>p_c$, then
\[
  \bnu_{p_n}\to\bnu_p
\]
weakly as probability measures on $\Om\times\partial\Gamma$.
\end{corollary}

\begin{proof}
Every subsequence has a weakly convergent subsequence in the compact space $\mathcal P_\mu(\partial\Gamma)$. By Lemma~\ref{lem:stationary-stability}, every such limit lies in $\mathcal S_p(\partial\Gamma)$, and by Theorem~\ref{thm:uniqueness-gromov} it must equal $\bnu_p$.
\end{proof}

\section{Continuity of the drift}\label{sec:drift}
In this section, we prove continuity of the drift using a Furstenberg-type formula for stationary families on the horoboundary. Similar formulas in random environments have appeared in \cite{LyonsPemantlePeres1995,CarrascoLessaPaquette2021}.

\begin{theorem}[Furstenberg-type formula]\label{thm:furstenberg}
Let $p>p_c$ and $\eta\in\mathcal S_p(\hbdG)$. Then
\begin{equation}\label{eq:furstenberg}
  \ell(p)=\frac1{Z_p}\sum_{s\in S} \int_{\Om\times\hbdG}a_{p,s}(\omega)\beta(s,\xi)\dd\eta(\omega,\xi).
\end{equation}
\end{theorem}

\begin{proof}
We sample $\bigl(\omega,(X_n)_{n\in\Z}\bigr)$ from $\Lambda_p$ and, conditionally on $\omega$, sample $\xi$ from $\eta_\omega$ independently of the path. We set
\[
  \omega_n=X_n^{-1}\omega, \quad \xi_n=X_n^{-1}\xi.
\]

The forward process $(\omega_n,\xi_n)_{n\ge0}$ is stationary. Indeed, for every bounded Borel function $F$ on $\Om\times\hbdG$, we have
\begin{align*}
  \sum_{s\in S} \int_\Om q_p(\omega,s) \int_{\hbdG} F(s^{-1}\omega,s^{-1}\xi) \dd\eta_\omega(\xi) \dd\whmu_p(\omega) &= \sum_{s\in S} \int_\Om q_p(\omega,s) \int_{\hbdG} F(\omega,s\xi) \dd\eta_{s^{-1}\omega}(\xi) \dd\whmu_p(\omega)\\
  &= \int_\Om\int_{\hbdG} F(\omega,\xi) \dd\eta_\omega(\xi) \dd\whmu_p(\omega).
\end{align*}
The first equality follows from reversibility, and the second from \eqref{eq:stationarity-pointwise}.

By the cocycle identity,
\[
  \beta(X_n,\xi) = \sum_{k=1}^n \beta(X_{k-1}^{-1}X_k,\xi_{k-1}).
\]
Conditionally on $(\omega_{k-1},\xi_{k-1})$, the increment $X_{k-1}^{-1}X_k$ equals $s$ with probability $q_p(\omega_{k-1},s)$, and hence
\begin{align}
  \frac1n\mathbb E[\beta(X_n,\xi)] &= \frac1n\sum_{k=1}^n \mathbb E\!\left[ \sum_{s\in S} q_p(\omega_{k-1},s)\beta(s,\xi_{k-1}) \right] \nonumber\\
  &= \int_\Om \sum_{s\in S}q_p(\omega,s) \int_{\hbdG} \beta(s,\xi)\dd\eta_\omega(\xi) \dd\whmu_p(\omega). \label{eq:furstenberg-expectation}
\end{align}

Conditional on $\omega$, the variables $\pi_h(\xi)$ and $\bnd^+(X)$ are independent, and the latter has the non-atomic law $\nu_\omega^p$ by Lemma~\ref{lem:harmonic-nonatomic}. Hence they are almost surely distinct. Lemma~\ref{lem:horofunction-asymptotics} and \eqref{eq:regularity} then give
\[
  \frac1n\beta(X_n,\xi)\to\ell(p)
\]
almost surely. Since
\[
  |\beta(X_n,\xi)|\le |X_n|\le n,
\]
dominated convergence shows that the left-hand side of \eqref{eq:furstenberg-expectation} tends to $\ell(p)$.

Finally, \eqref{eq:muhat} and \eqref{eq:qps} give
\[
  \begin{aligned}
    \int_\Om \sum_{s\in S}q_p(\omega,s) \int_{\hbdG} \beta(s,\xi)\dd\eta_\omega(\xi) \dd\whmu_p(\omega) = \frac1{Z_p} \sum_{s\in S} \int_{\Om\times\hbdG} a_{p,s}(\omega)\beta(s,\xi) \dd\eta(\omega,\xi).
  \end{aligned}
\]
This proves \eqref{eq:furstenberg}.
\end{proof}

The continuity of the drift follows from the above formula and Lemma~\ref{lem:stationary-stability}.

\begin{theorem}\label{thm:drift-continuity}
The function $p\mapsto\ell(p)$ is continuous on $(p_c,1]$.
\end{theorem}

\begin{proof}
Let $p_n\to p>p_c$. For each $n$, we choose $\eta_n\in\mathcal S_{p_n}(\hbdG)$. By compactness of $\mathcal P_\mu(\hbdG)$, every subsequence has a further subsequence, still denoted $(\eta_n)_{n\ge1}$, converging weakly to some $\eta$. Lemma~\ref{lem:stationary-stability} gives $\eta\in\mathcal S_p(\hbdG)$.

For fixed $s\in S$, the function $\beta(s,\cdot)$ is continuous and bounded. Since $a_{p_n,s}\to a_{p,s}$ in $L^1(\mu)$, Lemma~\ref{lem:weighted-weak-convergence} gives
\[
  \int a_{p_n,s}(\omega)\beta(s,\xi)\dd\eta_n(\omega,\xi) \to \int a_{p,s}(\omega)\beta(s,\xi)\dd\eta(\omega,\xi).
\]
Since $Z_{p_n}\to Z_p$, the Furstenberg formula at $p_n$ and $p$ gives $\ell(p_n)\to\ell(p)$ along this further subsequence. This proves continuity.
\end{proof}

\section{Continuity of the entropy}\label{sec:entropy}

In this section, we prove continuity of the entropy. Upper semicontinuity follows from continuity and subadditivity of the finite-time entropies. For lower semicontinuity, we establish a random-environment version of the Kaimanovich--Vershik formula \cite{KaimanovichVershik1983}, which expresses the entropy as an average of Kullback--Leibler divergences between hitting measures. The applicability of entropy theory to random walks in stationary random environments was already pointed out by Kaimanovich himself (see \cite{KaimanovichSobieczky2012}). We give a direct proof of the formula in the form needed here. We then combine this formula with continuity of the harmonic family (Corollary~\ref{cor:harmonic-family-continuity}) and lower semicontinuity of Kullback--Leibler divergence (Lemma~\ref{lem:relative-entropy-lsc}).

\subsection{Upper semicontinuity}
For a discrete random variable $Y$ under a probability measure $Q$, we write
\[
  H_Q(Y) = -\sum_y Q(Y=y)\log Q(Y=y)
\]
for its Shannon entropy. We set
\[
  H_n^\omega(p)=H_{P_p^\omega}(X_n), \quad H_n(p)=\int_\Om H_n^\omega(p)\dd\whmu_p(\omega).
\]
By Proposition~\ref{prop:bilateral},
\[
  -\frac1n\log P_p^\omega(X_n=x_n)\to h(p)
\]
for $\Lambda_p$-almost every $(\omega,x)$. Since every possible value of $X_n$ has quenched probability at least $|S|^{-n}$,
\[
  0\le -\frac1n\log P_p^\omega(X_n=x_n) \le \log|S|.
\]
Dominated convergence therefore gives
\[
  \frac{H_n(p)}{n}\to h(p).
\]
Moreover, the chain rule and stationarity of the environment process give
\[
  H_{m+n}(p)\le H_m(p)+H_n(p).
\]
Hence
\[
  h(p)=\inf_{n\ge1}\frac1nH_n(p).
\]

\begin{proposition}[Upper semicontinuity]\label{prop:entropy-usc}
If $p_n\to p>p_c$, then
\[
  \limsup_{n\to\infty}h(p_n)\le h(p).
\]
\end{proposition}

\begin{proof}
We extend the transition rule to every $\omega\in\Om$ by making the walk stay at its current vertex when no incident edge is open. We denote by $\widetilde H_m^\omega(q)$ the entropy of its position at time $m$. On $\Om_q^\infty$ this is $H_m^\omega(q)$, and
\[
  0 \le \widetilde H_m^\omega(q) \le m\log(|S|+1).
\]
For fixed $m$, $\widetilde H_m^\omega(q)$ depends only on the labels of finitely many edges near the root. Hence, if $p_n\to p$, then
\[
  \widetilde H_m^\omega(p_n) \to \widetilde H_m^\omega(p)
\]
for every $\omega$ outside the null set on which one of these labels equals $p$.

We write
\begin{align*}
  H_m(p_n)-H_m(p) &= \int_\Om \widetilde H_m^\omega(p_n)\dd(\whmu_{p_n}-\whmu_p)(\omega)\\
  &\quad+\int_\Om\bigl(\widetilde H_m^\omega(p_n)-\widetilde H_m^\omega(p)\bigr)\dd\whmu_p(\omega).
\end{align*}
The first term tends to zero by Lemma~\ref{lem:environment-continuity} and the uniform bound above, and the second by dominated convergence. Thus, for every $m\ge1$,
\[
  \limsup_{n\to\infty}h(p_n) \le \lim_{n\to\infty}\frac1mH_m(p_n) = \frac1mH_m(p).
\]
Taking the infimum over $m$ proves the claim.
\end{proof}

\subsection{Kaimanovich--Vershik formula in random environments}
We use the following convention for Kullback--Leibler divergence.

\begin{definition}
For finite Borel measures $\rho$ and $\sigma$ of equal mass, the \emph{Kullback--Leibler divergence} $D_{\mathrm{KL}}(\rho\,\|\,\sigma)$ is defined by
\[
  D_{\mathrm{KL}}(\rho\,\|\,\sigma) = \int \log\!\left(\frac{\dd\rho}{\dd\sigma}\right) \dd\rho
\]
when $\rho\ll\sigma$, and we set $D_{\mathrm{KL}}(\rho\|\sigma)=+\infty$ otherwise.
\end{definition}
Next, we introduce the conditional entropy with respect to the Gromov boundary. Let
\[
  P_p^\omega = \int_{\partial\Gamma} P_p^{\omega,\zeta}\dd\nu_\omega^p(\zeta)
\]
be the disintegration of the quenched path law with respect to the boundary map. Disintegrating the annealed law $\whmu_p(\dd\omega)P_p^\omega(\dd X)$ with respect to the map $(\omega,X)\mapsto\bigl(\omega,\bnd^+(X)\bigr)$ gives conditional laws $P_p^{\omega,\zeta}$ that are jointly measurable in $(\omega,\zeta)$.

We define the quenched and annealed conditional entropies by
\[
  H_n^{\omega,\partial\Gamma}(p) = \int_{\partial\Gamma} H_{P_p^{\omega,\zeta}}(X_n) \dd\nu_\omega^p(\zeta), \,\, H_n^{\partial\Gamma}(p) = \int_\Om H_n^{\omega,\partial\Gamma}(p) \dd\whmu_p(\omega).
\]

For $\whmu_p(\dd\omega)\nu_\omega^p(\dd\zeta)$-almost every $(\omega,\zeta)$, we set
\[
  q_p^\zeta(\omega,s) = P_p^{\omega,\zeta}(X_1=s).
\]
This is the one-step transition probability of the conditional walk. Bayes' formula gives
\begin{equation}\label{eq:boundary-bayes}
  q_p^\zeta(\omega,s) = q_p(\omega,s) \frac{\dd\nu_{\omega,s}^p}{\dd\nu_\omega^p}(\zeta)
\end{equation}
for $\whmu_p(\dd\omega)\nu_\omega^p(\dd\zeta)$-almost every $(\omega,\zeta)$ and every $s\in S$ such that $q_p(\omega,s)>0$. Recall that $\nu_{\omega,s}^p$ denotes the boundary law conditioned on the first step $s$.

The difference $H_n(p)-H_n^{\partial\Gamma}(p)$ measures the information that the boundary limit gives about $X_n$, conditional on the environment. The next proposition shows that this averaged quantity is additive in time.

\begin{proposition}
\label{prop:boundary-entropy-decomposition}
For every $n\ge1$,
\begin{equation}\label{eq:boundary-entropy-decomposition}
  H_n(p)-H_n^{\partial\Gamma}(p) = n\bigl( H_1(p)-H_1^{\partial\Gamma}(p) \bigr).
\end{equation}
\end{proposition}

\begin{proof}
All entropies are taken under $\Lambda_p$, whose forward marginal is $\whmu_p(\dd\omega)P_p^\omega(\dd X)$. We set
\[
  S_k=X_{k-1}^{-1}X_k, \quad \omega_k=X_k^{-1}\omega, \quad \mathcal F_k=\sigma(\omega,S_1,\ldots,S_k), \quad \zeta=\bnd^+(X).
\]

Since $X_n$ is determined by $(S_1,\ldots,S_n)$, the chain rule gives
\begin{align*}
  H(S_1,\ldots,S_n\mid\omega) &= H(X_n\mid\omega) + H(S_1,\ldots,S_n\mid\omega,X_n),\\
  H(S_1,\ldots,S_n\mid\omega,\zeta) &= H(X_n\mid\omega,\zeta) + H(S_1,\ldots,S_n\mid\omega,X_n,\zeta).
\end{align*}
Conditionally on $(\omega,X_n)$, the future from time $n$ is independent of $(S_1,\ldots,S_n)$, while $\zeta$ is determined by this future. Hence
\[
  H(S_1,\ldots,S_n\mid\omega,X_n,\zeta) = H(S_1,\ldots,S_n\mid\omega,X_n).
\]
Therefore
\begin{equation}\label{eq:entropy-difference}
  H_n(p)-H_n^{\partial\Gamma}(p) = H(S_1,\ldots,S_n\mid\omega) - H(S_1,\ldots,S_n\mid\omega,\zeta).
\end{equation}

For the first term, the chain rule and the Markov property give
\begin{align*}
  H(S_1,\ldots,S_n\mid\omega) &= \sum_{k=1}^n H(S_k\mid\mathcal F_{k-1})\\
  &= \sum_{k=1}^n \mathbb E\!\left[ -\sum_{s\in S} q_p(\omega_{k-1},s) \log q_p(\omega_{k-1},s) \right].
\end{align*}
Since $\omega_{k-1}$ has law $\whmu_p$ for every $k$,
\[
  H(S_1,\ldots,S_n\mid\omega) = nH_1(p).
\]

For the conditional term, we set $\zeta_k=X_k^{-1}\zeta$. The chain rule gives
\[
  H(S_1,\ldots,S_n\mid\omega,\zeta) = \sum_{k=1}^n H(S_k\mid\mathcal F_{k-1},\zeta).
\]
Since $X_{k-1}$ is $\mathcal F_{k-1}$-measurable,
\[
  \sigma(\mathcal F_{k-1},\zeta) = \sigma(\mathcal F_{k-1},\zeta_{k-1}).
\]
After rerooting at time $k-1$, the conditional law of $S_k$ given $\mathcal F_{k-1}$ and $\zeta$ is
\[
  q_p^{\zeta_{k-1}}(\omega_{k-1},\cdot).
\]
Consequently,
\begin{align*}
  H(S_k\mid\mathcal F_{k-1},\zeta) &= \mathbb E\!\left[ -\sum_{s\in S} q_p^{\zeta_{k-1}}(\omega_{k-1},s) \log q_p^{\zeta_{k-1}}(\omega_{k-1},s) \right].
\end{align*}
By $T$-invariance and equivariance of the boundary map,
\[
  (\omega_{k-1},\zeta_{k-1}) \sim \whmu_p(\dd\omega)\nu_\omega^p(\dd\zeta)
\]
for every $k$. Hence
\begin{align*}
  H(S_k\mid\mathcal F_{k-1},\zeta) &= -\int_\Om\int_{\partial\Gamma} \sum_{s\in S} q_p^\zeta(\omega,s) \log q_p^\zeta(\omega,s) \dd\nu_\omega^p(\zeta) \dd\whmu_p(\omega)\\
  &= H_1^{\partial\Gamma}(p).
\end{align*}
Therefore
\[
  H(S_1,\ldots,S_n\mid\omega,\zeta) = nH_1^{\partial\Gamma}(p).
\]
Substituting these expressions into \eqref{eq:entropy-difference} proves \eqref{eq:boundary-entropy-decomposition}.
\end{proof}

\begin{lemma}
\label{lem:boundary-conditional-entropy-zero}
For every $p>p_c$,
\[
  \lim_{n\to\infty} \frac1nH_{P_p^{\omega,\zeta}}(X_n) = 0
\]
for $\whmu_p(\dd\omega)\nu_\omega^p(\dd\zeta)$-almost every $(\omega,\zeta)$.
\end{lemma}

\begin{proof}
By sublinear tracking, for $\whmu_p(\dd\omega)\nu_\omega^p(\dd\zeta)$-almost every $(\omega,\zeta)$ and every $r>0$,
\[
  \delta_n = P_p^{\omega,\zeta}\!\left( d\bigl(X_n,\gamma_\zeta(\ell(p)n)\bigr)>rn \right) \to0.
\]
Since $G$ has bounded degree,
\[
  \log \left| B_G\bigl(\gamma_\zeta(\ell(p)n),rn\bigr) \right| \le rn\log|S|+O(1),
\]
while $X_n$ has at most $|S|^n$ possible values. Splitting according to whether $X_n$ lies in this ball gives
\[
  H_{P_p^{\omega,\zeta}}(X_n) \le rn\log|S| +\delta_n n\log|S| +O(1).
\]
Therefore
\[
  \limsup_{n\to\infty} \frac1n H_{P_p^{\omega,\zeta}}(X_n) \le r\log|S|.
\]
Letting $r\downarrow0$ proves the claim.
\end{proof}

\begin{theorem}[Kaimanovich--Vershik formula]
\label{thm:kv}
For every $p>p_c$,
\begin{equation}\label{eq:kv}
  h(p) = \int_\Om \sum_{s:\,q_p(\omega,s)>0} q_p(\omega,s) D_{\mathrm{KL}}\!\left( s\nu_{s^{-1}\omega}^p \,\middle\|\, \nu_\omega^p \right) \dd\whmu_p(\omega).
\end{equation}
\end{theorem}

\begin{proof}
By Lemma~\ref{lem:boundary-conditional-entropy-zero} and the bound
\[
  0\le H_{P_p^{\omega,\zeta}}(X_n)\le n\log |S|,
\]
dominated convergence gives
\[
  \frac1n H_n^{\partial\Gamma}(p)\to0.
\]
Dividing \eqref{eq:boundary-entropy-decomposition} by $n$ and letting $n\to\infty$ gives
\[
  h(p)=H_1(p)-H_1^{\partial\Gamma}(p).
\]
With the convention $0\log0=0$, the identity $\int_{\partial\Gamma}q_p^\zeta(\omega,s)\dd\nu_\omega^p(\zeta)=q_p(\omega,s)$ gives
\begin{align*}
  H_1(p)-H_1^{\partial\Gamma}(p) &= \int_\Om\int_{\partial\Gamma} \sum_{s:\,q_p(\omega,s)>0} q_p^\zeta(\omega,s) \log\frac{q_p^\zeta(\omega,s)}{q_p(\omega,s)} \dd\nu_\omega^p(\zeta)\dd\whmu_p(\omega)\\
  &= \int_\Om \sum_{s:\,q_p(\omega,s)>0} q_p(\omega,s) D_{\mathrm{KL}}\bigl(\nu_{\omega,s}^p\,\|\,\nu_\omega^p\bigr) \dd\whmu_p(\omega),
\end{align*}
where the second equality follows from \eqref{eq:boundary-bayes}. Since $\nu_{\omega,s}^p=s\nu_{s^{-1}\omega}^p$, this completes the proof.
\end{proof}

\begin{remark}
\label{rem:poisson-boundary}
By the entropy criterion for random walks along measured equivalence relations \cite[Theorem~2.17]{KaimanovichSobieczky2012}, for every $p>p_c$ and for $\mu_p^\infty$-almost every $\omega$, the Gromov boundary equipped with $\nu_\omega^p$ realizes the Poisson boundary of simple random walk on $C_p(\omega)$.
\end{remark}

\subsection{Lower semicontinuity}
Following Silva~\cite[Section~8]{Silva2026}, we use lower semicontinuity of Kullback--Leibler divergence. We apply it to joint measures on $\Om\times\partial\Gamma$.

\begin{lemma}
\label{lem:relative-entropy-lsc}
Let $Y$ be a compact metrizable space. Suppose that $\rho_n,\sigma_n$ are finite Borel measures on $Y$, $\rho_n\to\rho$ and $\sigma_n\to\sigma$ weakly, and $\rho_n(Y)=\sigma_n(Y)$ for every $n$. Then
\[
  \liminf_{n\to\infty} D_{\mathrm{KL}}(\rho_n\,\|\,\sigma_n) \ge D_{\mathrm{KL}}(\rho\,\|\,\sigma).
\]
\end{lemma}

\begin{proof}
For finite measures of equal mass, the variational formula \cite[Section~6.2]{DemboZeitouni1998} gives
\[
  D_{\mathrm{KL}}(\rho\,\|\,\sigma) = \sup_{f\in C(Y)} \left\{ \int f\dd\rho - \int(e^f-1)\dd\sigma \right\}.
\]
Each expression inside the supremum is continuous under weak convergence, so the supremum is lower semicontinuous.
\end{proof}

For $s\in S$, we define finite measures on $\Om\times\partial\Gamma$ by
\begin{align*}
  M_{p,s}(\dd\omega,\dd\zeta) &= a_{p,s}(\omega)\mu(\dd\omega) \bigl( s\nu_{s^{-1}\omega}^p \bigr)(\dd\zeta), \\
  N_{p,s}(\dd\omega,\dd\zeta) &= a_{p,s}(\omega)\mu(\dd\omega) \nu_\omega^p(\dd\zeta).
\end{align*}
They have the same first marginal $a_{p,s}\mu$, and hence the same total mass.

\begin{proposition}[Lower semicontinuity]
\label{prop:entropy-lsc}
If $p_n\to p>p_c$, then
\[
  \liminf_{n\to\infty}h(p_n)\ge h(p).
\]
\end{proposition}

\begin{proof}
Since $M_{p,s}$ and $N_{p,s}$ have the same first marginal, the chain rule gives
\[
  D_{\mathrm{KL}}\bigl(M_{p,s}\,\|\,N_{p,s}\bigr) = \int_\Om a_{p,s}(\omega) D_{\mathrm{KL}}\!\left( s\nu_{s^{-1}\omega}^p\,\middle\|\,\nu_\omega^p \right)\dd\mu(\omega).
\]
Together with $\whmu_p(\dd\omega)q_p(\omega,s)=Z_p^{-1}a_{p,s}(\omega)\mu(\dd\omega)$ and \eqref{eq:kv}, this yields
\begin{equation}\label{eq:finite-measure-kv}
  h(p) = \frac1{Z_p} \sum_{s\in S} D_{\mathrm{KL}} \bigl( M_{p,s}\,\|\,N_{p,s} \bigr).
\end{equation}

By Corollary~\ref{cor:harmonic-family-continuity}, we have
\[
  \bnu_{p_n}\to\bnu_p
\]
weakly.
Since $a_{p_n,s}\to a_{p,s}$ in $L^1(\mu)$, Lemma~\ref{lem:weighted-weak-convergence} gives
\[
  N_{p_n,s}\to N_{p,s}.
\]
For $M_{p_n,s}$, we use the homeomorphism $\Psi_s(\omega,\zeta)=(s\omega,s\zeta)$ of $\Om\times\partial\Gamma$. The measure $(\Psi_s)_*\bnu_p$ has first marginal $\mu$ and conditional measures $s\nu_{s^{-1}\omega}^p$. Hence
\[
  (\Psi_s)_*\bnu_{p_n} \to (\Psi_s)_*\bnu_p,
\]
and another application of Lemma~\ref{lem:weighted-weak-convergence} yields
\[
  M_{p_n,s}\to M_{p,s}.
\]

Applying Lemma~\ref{lem:relative-entropy-lsc} to \eqref{eq:finite-measure-kv} and using the finiteness of $S$ and $Z_{p_n}\to Z_p$, we obtain
\[
  \liminf_{n\to\infty}h(p_n) \ge \frac1{Z_p} \sum_{s\in S} D_{\mathrm{KL}} \bigl( M_{p,s}\,\|\,N_{p,s} \bigr) = h(p).\qedhere
\]
\end{proof}

\begin{proof}[Proof of Theorem~\ref{thm:main}]
Continuity of the drift $\ell$ follows from Theorem~\ref{thm:drift-continuity}. Continuity of the entropy $h$ follows from Propositions~\ref{prop:entropy-usc} and~\ref{prop:entropy-lsc}.
\end{proof}

\begin{proof}[Proof of Corollary~\ref{cor:dimension-intro}]
By \cite[Theorem~1.1]{Sakamoto2024}, the dimension is given by \eqref{eq:dimension-formula-intro}. Since $\ell(p)>0$ for $p>p_c$, the claim follows from Theorem~\ref{thm:main}.
\end{proof}

\bibliographystyle{amsalpha}
\bibliography{continuity_references}
\enlargethispage{\baselineskip}

\end{document}